\documentclass[12pt]{amsart}
\usepackage{amsfonts,amssymb,amscd,amsmath,enumerate,verbatim}
\usepackage[latin1]{inputenc}
\usepackage{amscd}
\usepackage{latexsym}
\usepackage{mathptmx}
\usepackage{multicol}
\input xy
\xyoption{all}

\def\frk{\mathfrak}               

\def\Phi{{\frk N}}
\def\opn#1#2{\def#1{\operatorname{#2}}} 
\opn\chara{char} 
\opn\length{\ell} 
\opn\pd{pd} 
\opn\rk{rk}
\opn\projdim{proj\,dim} 
\opn\injdim{inj\,dim} 
\opn\rank{rank}
\opn\depth{depth} 
\opn\grade{grade} 
\opn\height{height}
\opn\embdim{emb\,dim} 
\opn\codim{codim}

\opn\Tr{Tr} 
\opn\bigrank{big\,rank}
\opn\superheight{superheight}
\opn\lcm{lcm}
\opn\trdeg{tr\,deg}
\opn\reg{reg} 
\opn\lreg{lreg} 
\opn\ini{in} 
\opn\lpd{lpd}
\opn\size{size}
\opn\mult{mult}
\opn\dist{dist}
\opn\cone{cone}
\opn\lex{lex}
\opn\rev{rev}
\opn\im{im}
\opn\m{m}
\opn\div{div} \opn\Div{Div} \opn\cl{cl} \opn\Cl{Cl}
\opn\Spec{Spec} \opn\Supp{Supp} \opn\supp{supp} \opn\Sing{Sing}
\opn\Ass{Ass} \opn\Min{Min}
\opn\Ann{Ann} \opn\Rad{Rad} \opn\Soc{Soc}
\opn\Syz{Syz} \opn\Im{Im} \opn\Ker{Ker} \opn\Coker{Coker}
\opn\Am{Am} \opn\Hom{Hom} \opn\Tor{Tor} \opn\Ext{Ext}
\opn\End{End} \opn\Aut{Aut} \opn\id{id} \opn\ini{in}

\opn\nat{nat}
\opn\pff{pf}
\opn\Pf{Pf} \opn\GL{GL} \opn\SL{SL} \opn\mod{mod} \opn\ord{ord}
\opn\Gin{Gin}
\opn\Hilb{Hilb}\opn\adeg{adeg}\opn\std{std}\opn\ip{infpt}
\opn\Pol{Pol}
\opn\sat{sat}
\opn\Var{Var}
\opn\Gen{Gen}

\opn\aff{aff} \opn\con{conv} \opn\relint{relint} \opn\st{st}
\opn\lk{lk} \opn\cn{cn} \opn\core{core} \opn\vol{vol}
\opn\link{link} \opn\star{star}
\opn\gr{gr}

\def\pot#1#2{#1[\kern-0.28ex[#2]\kern-0.28ex]}

\opn\dirlim{\underrightarrow{\lim}}
\opn\inivlim{\underleftarrow{\lim}}
\def\Implies{\ifmmode\Longrightarrow \else
        \unskip${}\Longrightarrow{}$\ignorespaces\fi}
\def\implies{\ifmmode\Rightarrow \else
        \unskip${}\Rightarrow{}$\ignorespaces\fi}
\def\iff{\ifmmode\Longleftrightarrow \else
        \unskip${}\Longleftrightarrow{}$\ignorespaces\fi}

\let\:=\colon
\newtheorem{Theorem}{Theorem}[section]
\newtheorem{Lemma}[Theorem]{Lemma}

\newtheorem{Remark}[Theorem]{Remark}

\newtheorem{Example}[Theorem]{Example}

\let\epsilon\varepsilon
\let\phi=\varphi
\let\kappa=\varkappa
\def\qed{\ifhmode\textqed\fi
      \ifmmode\ifinner\quad\qedsymbol\else\dispqed\fi\fi}
\def\textqed{\unskip\nobreak\penalty50
       \hskip2em\hbox{}\nobreak\hfil\qedsymbol
       \parfillskip=0pt \finalhyphendemerits=0}
\def\dispqed{\rlap{\qquad\qedsymbol}}

\opn\dis{dis}
\opn\height{height}
\opn\dist{dist}
\def\pnt{{\raise0.5mm\hbox{\large\bf.}}}

\opn\Lex{Lex}

\begin{document}

\title{Matching numbers and dimension of edge ideals}
\author{Ayana Hirano and Kazunori Matsuda}

\address{Ayana Hirano, 
Kitami Institute of Technology, 
Kitami, Hokkaido 090-8507, Japan}
\email{f1510801010@mail.kitami-it.ac.jp}

\address{Kazunori Matsuda,
Kitami Institute of Technology, 
Kitami, Hokkaido 090-8507, Japan}
\email{kaz-matsuda@mail.kitami-it.ac.jp}
\subjclass[2010]{05C69, 05C70, 05E40, 13C15}
\keywords{edge ideal, induced matching number, minimum matching number, matching number, independent set.}
\begin{abstract}
Let $G$ be a finite simple graph on the vertex set $V(G) = \{x_{1}, \ldots, x_{n}\}$ and match$(G)$, min-match$(G)$ and ind-match$(G)$ the matching number, minimum matching number and induced matching number of $G$, respectively. 
Let $K[V(G)] = K[x_{1}, \ldots, x_{n}]$ denote the polynomial ring over a field $K$ and $I(G) \subset K[V(G)]$ the edge ideal of $G$. 
The relationship between these graph-theoretic invariants and ring-theoretic invariants of the quotient ring $K[V(G)]/I(G)$ 
has been studied. 
In the present paper, we study the relationship between match$(G)$, min-match$(G)$, ind-match$(G)$ 
and $\dim K[V(G)]/I(G)$. 

\end{abstract}
\maketitle
\section*{Introduction}
    

Let $G = (V(G), E(G))$ be a finite simple graph (i.e. a finite graph with no loop and no multiple edge) 
on the vertex set $V(G) = \{x_{1}, \ldots, x_{n}\}$ with the edge set $E(G)$.   
%
A subset $M = \{e_{1}, \ldots, e_{s} \} \subset E(G)$ is said to be a {\em matching} of $G$ if, for all $e_{i}$ and $e_{j}$ with $i \neq j$ belonging to $M$, one has $e_i \cap e_j = \emptyset$. 
A matching $M$ of $G$ is {\em maximal} if $M \cup \{e\}$ cannot be a matching of $G$ for all $e \in E(G) \setminus M$.    
A matching $M = \{e_{1}, \ldots, e_{s} \} \subset E(G)$ is said to be an {\em induced matching} of $G$ if for all $e_{i}$ and $e_{j}$ with $i \neq j$ belonging to $M$, there is no edge $e \in E(G)$ with $e \cap e_{i} \neq \emptyset$ and $e \cap e_{j} \neq \emptyset$. 
The {\em matching number} match$(G)$, the {\em minimum matching number} min-match$(G)$ 
and the {\em induced matching number} ind-match$(G)$ of $G$ are defined as follows respectively: 
\begin{eqnarray*}
\text{match}(G)&=&\max\{|M| : M \text{\ is a matching of\ } G\}; \\
\text{min-match}(G)&=&\min\{|M| : M \text{\ is a maximal matching of\ } G\}; \\
\text{ind-match}(G)&=&\max\{|M| : M \text{\ is an induced matching of\ } G\}.  
\end{eqnarray*}
It is known that ind-match$(G)$ $\leq$ min-match$(G)$ $\leq$ match$(G)$ $\leq$ 2min-match$(G)$ holds 
for all finite simple graph $G$ (\cite[Proposition 2.1]{HHKT}). 
It is also known that there exists a finite connected simple graph $G_{a, b, c}$ 
with ind-match$(G_{a, b, c})$ $= a$, min-match$(G_{a, b, c})$ $= b$ and match$(G_{a, b, c}) = c$ 
for all positive integers $a, b, c$ satisfying $1 \leq a \leq b \leq c \leq 2b$  (\cite[Theorem 2.3]{HHKT}). 
A classification of finite connected simple graphs $G$ with ind-match$(G)$ $=$ min-match$(G)$ $=$ match$(G)$ 
is given (\cite[Theorem 1]{CW}, \cite[Remark 0.1]{HHKO}). 
Such graphs are studied in \cite{HHKO, HKMT, T} from a viewpoint of commutative algebra. 
In addition, finite simple graphs $G$ with ind-match$(G)$ $=$ min-match$(G)$ are also studied in \cite{HHKT}.  

Let $G = (V(G), E(G))$ be a finite simple graph on the vertex set $V(G) = \{x_{1}, \ldots, x_{n}\}$ with the edge set $E(G)$.   
Let $K[V(G)] = K[x_{1}, \ldots, x_{n}]$ be the polynomial ring in $n = |V(G)|$ variables over a field $K$ with each $\deg x_{i} = 1$.  
The {\em edge ideal} of $G$ is the ideal 
\[
I(G) = \left( x_{i}x_{j} : \{x_{i}, x_{j}\} \in E(G) \right) \subset K[V(G)]. 
\] 
Edge ideals of finite simple graphs have been studied by many researcher, see \cite{HV, MV, V} and their references. 
Let ${\rm reg}(G) = {\rm reg}\left(K[V(G)]/I(G)\right)$, $\deg h(G) = \deg h\left(K[V(G)]/I(G)\right)$ and $\dim (G) = \dim K[V(G)]/I(G)$ 
denote the regularity, the degree of the $h$-polynomial (see \cite[p.312]{HKM}) and the dimension 
of the quotient ring $K[V(G)]/I(G)$, respectively. 
In general, $\deg h(G) \leq \dim (G)$ holds. 
Moreover, ${\rm reg}(G) = \deg h(G)$ holds if $K[V(G)]/I(G)$ is Cohen-Macaulay (\cite[Corollary B.4.1]{V}). 
It is known that there exists a finite connected simple graph $G_{r, s}$ with ${\rm reg}(G_{r, s}) = r$ and $\deg h(G_{r, s}) = s$ 
for all integers $r, s \geq 1$ (\cite[Theorem 3.1]{HMVT}).  

Recently, the relationship between graph-theoretic invariants match$(G)$, min-match$(G)$, ind-match$(G)$ 
and ring-theoretic invariants of the quotient ring $K[V(G)]/I(G)$ has been studied. 
As previous results, 
\begin{itemize}
	\item \cite{K, W} ind-match$(G) \leq {\rm reg} (G) \leq$ min-match$(G) \leq$ match$(G)$;  
	\item \cite[Theorem 11]{T} there is no finite connected simple graph $G$ with ind-match$(G) = 1$ and 
	${\rm reg} (G) =$ min-match$(G) =$ match$(G) = r$ for all $r \geq 3$; 
	\item \cite[Theorem 0.1]{HKM} there exists a finite connected simple graph $G_{a, r, s}$ satisfying  
	ind-match $(G_{a, r, s}) = a$, ${\rm reg}(G_{a, r, s}) = r$ and $\deg h(G_{a, r, s}) = s$ for all integers $a, r, s$ 
	with $1 \leq a \leq r$ and $s \geq 1$.  
\end{itemize}

In the present paper, we study the relationship between match$(G)$, min-match$(G)$, ind-match$(G)$ 
and $\dim K[V(G)]/I(G)$.
We will prove that 

\begin{Theorem}\label{main}
Let $a, b, c, d$ be positive integers. 
Then the following assertions are equivalent:
\begin{enumerate}
	\item[$(1)$] there exists a finite connected simple graph $G_{a, b, c, d}$ with ${\rm ind}$-${\rm match}(G_{a, b, c, d}) = a$, 
	${\rm min}$-${\rm match}(G_{a, b, c, d}) = b$, ${\rm match}(G_{a, b, c, d}) = c$ and $\dim (G_{a, b, c, d}) = d$; 
	\item[$(2)$] $1 \leq a \leq b \leq c \leq 2b$ and $d \geq \max\{a, 2(c - b)\}$.  
\end{enumerate}
\end{Theorem}

\section{Preparation for Theorem \ref{main}}

In order to prove Theorem \ref{main}, we will prepare several lemmata in this section. 

Let $G = (V(G), E(G))$ be a finite simple graph on the vertex set $V(G) = \{x_{1}, \ldots, x_{n}\}$ with the edge set $E(G)$.  
Given any subset $W \subset V(G)$, the {\em induced subgraph} of $G$ on $W$, denoted by $G_{W}$, 
is the graph on $V(G_{W}) = W$ with $E(G_{W}) = \left\{ \{x_{i}, x_{j}\} : i, j \in W \right\}$. 
A subset $S \subset V(G)$ is an {\em independent set} of $G$ if $\{x_{i}, x_{j}\} \not\in E(G)$ for all $x_{i}, x_{j} \in S$. 
Note that the empty set $\emptyset$ is an independent set of $G$. 
It is known that 

\begin{Lemma}\label{dim(G)}
$\dim (G) = \max\{ |S| : S \text{\ is\ an\ independent\ set\ of\ } G \}$. 
\end{Lemma}

Let us recall the definition of $S$-suspension (see \cite[p.313]{HKM}) of $G$. 
Let $S \subset V(G)$ be an independent set of $G$. 
Note that $0 \leq |S| \leq \dim (G)$ by Lemma \ref{dim(G)}.   
The graph $G^{S}$ is defined by
\begin{itemize}
	\item $V(G^{S}) = V(G) \cup \{ x_{n + 1} \}$, where $x_{n + 1}$ is a new vertex. 
	\item $E(G^{S}) = E(G) \cup \left\{ \{ x_{i}, x_{n + 1} \} : x_{i} \not\in S \right\}$. 
\end{itemize}
We call $G^{S}$ the $S$-{\em suspension} of $G$. 
Note that $G^{\emptyset}$ coincides with the suspension \cite[p.141]{HNOS} of $G$ in usual sense. 
The following lemma mentions the induced matching number and the dimension of 
the $S$-suspension of a graph. 

\begin{Lemma}[{\cite[Lemma 1.5]{HKM}}]
\label{S-suspension}
Let $G$ be a finite simple graph on the vertex set $V(G) = \{x_{1}, \ldots, x_{n}\}$
such that $G$ has no isolated vertices. 
Let $S \subset V(G)$ be an independent set of $G$. 
Then one has 
\begin{enumerate}
	\item[$(1)$]  ${\rm ind}$-${\rm match}(G^{S}) = {\rm ind}$-${\rm match}(G)$. \\
	\item[$(2)$] $\dim (G^{S}) = \begin{cases} \displaystyle \dim (G) & \text{$(0 \leq |S| \leq \dim (G) - 1)$}, \\ \displaystyle \dim (G) + 1 & \text{$(|S| = \dim (G))$}. \end{cases}$ 
\end{enumerate}
\end{Lemma}


\begin{Remark}\normalfont
In general, the $S$-suspension does not preserve the matching number and the minimum matching number. 
In fact, let $C_{3}$ be the 3-cycle on $V(C_{3}) = \{ x_{1}, x_{2}, x_{3}\}$ with 
$E(C_{3}) = \left\{ \{x_{1}, x_{2}\}, \{x_{2}, x_{3}\}, \{x_{1}, x_{3}\} \right\}$. 
Then match$(C_{3}) =$ min-match$(C_{3}) = 1$ and match$(C_{3}^{\emptyset}) =$ min-match$(C_{3}^{\emptyset}) = 2$. 
\end{Remark}


Next, we give two lower bounds of $\dim (G)$. 

\begin{Lemma}\label{inequalities} 
Let $G = (V(G), E(G))$ be a finite simple graph. Then
\begin{enumerate}
	\item[$(1)$] $\dim (G) \geq {\rm ind}$-${\rm match}(G)$.  
	\item[$(2)$] $\dim (G) \geq 2({\rm match}(G) - {\rm min}$-${\rm match}(G))$. 
\end{enumerate}
\end{Lemma}
\begin{proof}
(1) Assume ${\rm ind}$-${\rm match}(G) = a$.  
Let $\left\{ \{x_{1}, x_{a + 1}\}, \{x_{2}, x_{a + 2}\}, \ldots, \{x_{a}, x_{2a}\} \right\} \subset E(G)$ be an induced matching of $G$. 
Then $\{ x_{1}, x_{2}, \ldots, x_{a} \} \subset V(G)$ is an independent set of $G$. 
Hence one has $\dim (G) \geq a$ from Lemma \ref{dim(G)}. 

(2) Assume ${\rm min}$-${\rm match}(G) = b$ and ${\rm match}(G) = c$. Then we have $|V(G)|\geq 2c$. 
Let $\left\{ \{x_{1}, x_{b + 1}\}, \{x_{2}, x_{b + 2}\}, \ldots, \{x_{b}, x_{2b}\} \right\} \subset E(G)$ be a minimum matching of $G$ 
and put $W = \{x_{1}, x_{2}, \ldots, x_{2b}\}$. 
Then the induced subgraph $G_{V(G) \setminus W}$ has no edge by definition of minimum matching. 
Hence $V(G_{V(G) \setminus W})$ is an independent set of $G$. 
Since $|V(G)|\geq 2c$ and $|W| = 2b$, it follows that $|V(G_{V(G) \setminus W})| \geq 2(c - b)$. 
Therefore $\dim (G) \geq 2(c - b)$.  
\end{proof}

Let $s \geq 1$ be an integer. 
Let $G^{{\rm star}(x_{v})}_{s}$ be the graph on $V(G^{{\rm star}(x_{v})}_{s}) = \{ x_{1}, \ldots, x_{s}, x_{v} \}$ with 
$E(G^{{\rm star}(x_{v})}_{s}) = \left\{ \{x_{i}, x_{v}\} : 1 \leq i \leq s \right\}$; see \cite[Figure 2]{HKMT}.  
We call $G^{{\rm star}(x_{v})}_{s}$ the {\em star graph}. 
The star graph is the complete bipartite graph $K_{1, s}$. 
%
%
%
The {\em complete graph} $K_{n}$ is the graph on the vertex set $V(K_{n}) = \{x_{1}, \ldots, x_{n}\}$ with 
its edge set $E(K_{n}) = \left\{ \{x_{i}, x_{j}\} : 1 \leq i < j \leq n \right\}$. 
Lemma \ref{dim(G)} says that $\dim (G) = 1$ if and only if $G = K_{n}$. 
The following lemma mentions invariants of star graphs and complete graphs.  

\begin{Lemma}\label{starK2n}
Let $s \geq 1$ be an integer. 
\begin{enumerate}
	\item[$(1)$] Let $G^{{\rm star}(x_{v})}_{s}$ be the star graph which appears as above. 
	Then ${\rm ind}$-${\rm match}(G^{{\rm star}(x_{v})}_{s}) = {\rm min}$-${\rm match}(G^{{\rm star}(x_{v})}_{s}) = {\rm match}(G^{{\rm star}(x_{v})}_{s}) = 1$ and $\dim (G^{{\rm star}(x_{v})}_{s}) = s$. 
	\item[$(2)$] Let $K_{2s}$ be the complete graph on $2s$ vertices. 
	Then ${\rm ind}$-${\rm match}(K_{2s}) = \dim (K_{2s}) = 1$ and ${\rm min}$-${\rm match}(K_{2s}) = {\rm match}(K_{2s}) = s$. 
\end{enumerate}
\end{Lemma}

\begin{Lemma}\label{inducedsubgraph}
Let $G$ be a finite simple graph on the vertex set $V(G)$ and $W \subset V(G)$ a subset. 
Then one has 
\begin{enumerate}
	\item[$(1)$] ${\rm match}(G_{W}) \leq {\rm match}(G)$. 
	\item[$(2)$] ${\rm min}$-${\rm match}(G_{W}) \leq {\rm min}$-${\rm match}(G)$.  
	\item[$(3)$] ${\rm ind}$-${\rm match}(G_{W}) \leq {\rm ind}$-${\rm match}(G)$. 
\end{enumerate}
\end{Lemma}

\begin{Lemma}\label{match}
Let $G = (V(G), E(G))$ be a finite simple graph. 
\begin{enumerate}
	\item[$(1)$] Assume that there exist two edges $\{x_{i}, x_{j}\}, \{x_{i}, x_{k}\} \in E(G)$ 
	such that $\deg (x_{j}) = \deg (x_{k}) = 1$. 
	Then 
	\begin{enumerate}
		\item ${\rm match}(G) = {\rm match}(G_{V(G) \setminus \{x_{k}\}})$.
		\item ${\rm min}$-${\rm match}(G) = {\rm min}$-${\rm match}(G_{V(G) \setminus \{x_{k}\}})$.
		\item ${\rm ind}$-${\rm match}(G) = {\rm ind}$-${\rm match}(G_{V(G) \setminus \{x_{k}\}})$.
	\end{enumerate}
	\item[$(2)$] ${\rm match}(G) \leq \left\lfloor \frac{|V(G)|}{2} \right\rfloor$. 
\end{enumerate} 
\end{Lemma}

\begin{Lemma}
\label{disconnected}
Let $G$ be a finite disconnected simple graph and $G_{1}, \ldots, G_{\ell}$ the connected components of $G$. 
Then we have 
\begin{enumerate}
	\item[$(1)$] ${\rm match}(G) = \sum_{i = 1}^{\ell} {\rm match}(G_{i})$. 
	\item[$(2)$] ${\rm min}$-${\rm match}(G) = \sum_{i = 1}^{\ell} {\rm min}$-${\rm match}(G_{i})$.  
	\item[$(3)$] ${\rm ind}$-${\rm match}(G) = \sum_{i = 1}^{\ell} {\rm ind}$-${\rm match}(G_{i})$. 
	\item[$(4)$] $\dim (G) = \sum_{i = 1}^{\ell} \dim (G_{i})$. 
\end{enumerate}
\end{Lemma}


\section{Proof of Theorem \ref{main}}

In this section, we give a proof of Theorem \ref{main}. 

\begin{proof}({\em Proof of Theorem \ref{main}.})

(1) $\Rightarrow$ (2) : It follows from \cite[Proposition 2.1]{HHKT} and Lemma \ref{inequalities}. 

(2) $\Rightarrow$ (1) : Let $a, b, c, d$ be positive integers satisfying $1 \leq a \leq b \leq c \leq 2b$ and $d \geq \max\{a, 2(c - b)\}$. 
Put $V_{2b} = \{v_{1}, v_{2}, \ldots, v_{2b}\}$. 

$\bullet \ ${\bf Case 1: $a = 1, 1 \leq b = c$ and $d \geq 1. $} \\ 
Let $G_{1, b, b, d}^{(1)}$ be the graph such that 
\begin{eqnarray*}
V(G_{1, b, b, d}^{(1)}) &=& V_{2b} \cup \{x_{1}, x_{2}, \ldots, x_{d - 1}\} \\
E(G_{1, b, b, d}^{(1)}) &=& \left\{ \{v_{i}, v_{j}\} : 1 \leq i < j \leq 2b \right\} \cup \left\{ \{v_{1}, x_{k}\} : 1 \leq k \leq d - 1 \right\}; 
\end{eqnarray*}
see Figure \ref{fig:G_{1, b, b, d}^{(1)}}. 

\begin{figure}[htbp]
\centering
\bigskip

\begin{xy}
	\ar@{} (0,0);(75, -12)  *++! U{v_{1}} *\cir<4pt>{} = "C"
	\ar@{-} "C";(55, 0) *++! D{x_{1}} *\cir<4pt>{} = "D";
	\ar@{-} "C";(65, 0) *++!D{x_{2}} *\cir<4pt>{} = "E";
	\ar@{} "C"; (79, 3.7) *++!U{\cdots}
	\ar@{-} "C";(95, 0) *++!D{x_{d - 1}} *\cir<4pt>{} = "F";
	\ar@{} (0,0);(75,-21.9) *\cir<25pt>{};
	\ar@{} (0,0);(75,-22.5) *{\text{$K_{2b}$}};
\end{xy}


  \caption{The graph $G_{1, b, b, d}^{(1)}$}
  \label{fig:G_{1, b, b, d}^{(1)}}
\end{figure}
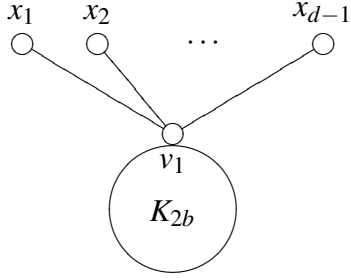

Then ${\rm ind}$-${\rm match}(G_{1, b, b, d}^{(1)}) = 1$ since each edge of $G_{1, b, b, d}^{(1)}$ contains 
a vertex belonging to $V_{2b}$. 
We prove ${\rm min}$-${\rm match}(G_{1, b, b, d}^{(1)}) = {\rm match}(G_{1, b, b, d}^{(1)}) = b$. 
If $d = 1$, then $G_{1, b, b, 1}^{(1)} = K_{2b}$. 
Hence ${\rm min}$-${\rm match}(G_{1, b, b, 1}^{(1)}) = {\rm match}(G_{1, b, b, 1}^{(1)}) = b$ by Lemma \ref{starK2n}. 
Assume $d > 1$. 
Since $K_{2b}$ is an induced subgraph of $(G_{1, b, b, d}^{(1)})$,  
by using Lemma \ref{starK2n}, \ref{inducedsubgraph} and \ref{match}, we have
$b = {\rm min}$-${\rm match}(K_{2b}) \leq {\rm min}$-${\rm match}(G_{1, b, b, d}^{(1)}) \leq {\rm match}(G_{1, b, b, d}^{(1)}) = {\rm match}(G_{1, b, b, 2}^{(1)}) \leq b. $
Thus we have ${\rm min}$-${\rm match}(G_{1, b, b, d}^{(1)}) = {\rm match}(G_{1, b, b, d}^{(1)}) = b$. 
A subset $\{ x_{1}, \ldots, x_{d - 1}, v_{2b} \} \subset V(G_{1, b, b, d}^{(1)})$ is an independent set of $G_{1, b, b, d}^{(1)}$. 
Hence $\dim (G_{1, b, b, d}^{(1)}) \geq d$ by Lemma \ref{dim(G)}. 
Let $W \subset V(G_{1, b, b, d}^{(1)})$ with $|W| \geq d + 1$. 
Then $|W \cap V_{2b}| \geq 2$. Thus $W$ is not an independent set. 
Therefore one has $\dim (G_{1, b, b, d}^{(1)}) = d$. 

$\bullet \ ${\bf Case 2: $a = 1, 1 \leq b < c \leq 2b$ and $d = 2(c - b). $} \\ 
Let $G_{1, b, c, 2(c - b)}^{(2)}$ be the graph such that 
\begin{eqnarray*}
V(G_{1, b, c, 2(c - b)}^{(2)}) &=& V_{2b} \cup \{x_{1}, x_{2}, \ldots, x_{2(c - b)}\}, \\
E(G_{1, b, c, 2(c - b)}^{(2)}) &=& \left\{ \{v_{i}, v_{j}\} : 1 \leq i < j \leq 2b \right\} \cup \bigcup_{k = 1}^{b} \bigcup_{\ell = 1}^{c - b} \left\{ \{v_{k}, x_{\ell}\}, \{ v_{b + k}, x_{c - b + \ell} \} \right\};
\end{eqnarray*}
see Figure \ref{fig:G_{1, b, c, 2(c - b)}^{(2)}}. 

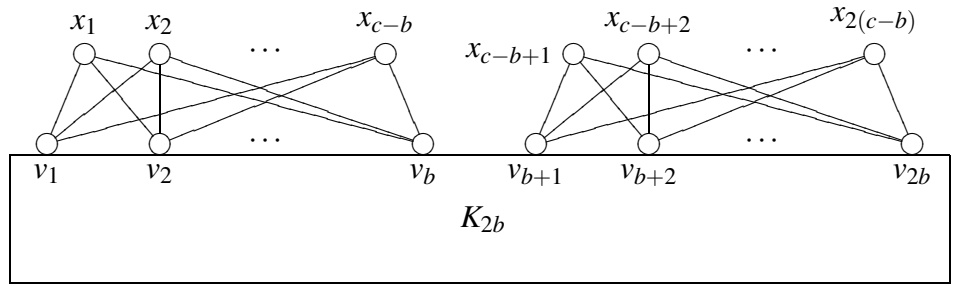
\begin{figure}[htbp]
\centering

\begin{xy}
	\ar@{} (0,0);(20, -16)  *++! U{v_{1}} *\cir<4pt>{} = "V1"
	\ar@{} (0,0);(35, -16)  *++! U{v_{2}} *\cir<4pt>{} = "V2"
	\ar@{} (0,0);(70, -16)  *++! U{v_{b}} *\cir<4pt>{} = "Vb"
	\ar@{-} "V1";(25, -4) *++! D{x_{1}} *\cir<4pt>{} = "X1";
	\ar@{-} "V1";(35, -4) *++!D{x_{2}} *\cir<4pt>{} = "X2";
	\ar@{} "V1"; (49, 0) *++!U{\cdots}
	\ar@{} "V1"; (49, -12) *++!U{\cdots}
	\ar@{-} "V1";(65, -4) *++!D{x_{c - b}} *\cir<4pt>{} = "Xc-b";
	\ar@{-} "V2"; "X1";
	\ar@{-} "V2"; "X2";
	\ar@{-} "V2"; "Xc-b";
	\ar@{-} "Vb"; "X1";
	\ar@{-} "Vb"; "X2";
	\ar@{-} "Vb"; "Xc-b";
	\ar@{} (0,0);(85, -16)  *++! U{v_{b + 1}} *\cir<4pt>{} = "Vb+1";
	\ar@{} (0,0);(100, -16)  *++! U{v_{b + 2}} *\cir<4pt>{} = "Vb+2";
	\ar@{} (0,0);(135, -16)  *++! U{v_{2b}} *\cir<4pt>{} = "V2b";
	\ar@{-} "Vb+1";(90, -4) *++! R{x_{c - b + 1}} *\cir<4pt>{} = "Xc-b+1";
	\ar@{-} "Vb+1";(100, -4) *++!D{x_{c - b + 2}} *\cir<4pt>{} = "Xc-b+2";
	\ar@{-} "Vb+1";(130, -4) *++!D{x_{2(c - b)}} *\cir<4pt>{} =  "X2(c-b)";
	\ar@{-} "Vb+2"; "Xc-b+1";
	\ar@{-} "Vb+2"; "Xc-b+2";
	\ar@{-} "Vb+2"; "X2(c-b)";
	\ar@{-} "V2b"; "Xc-b+1";
	\ar@{-} "V2b"; "Xc-b+2";
	\ar@{-} "V2b"; "X2(c-b)";
	\ar@{} "V1"; (115, 0) *++!U{\cdots}
	\ar@{} "V1"; (115, -12) *++!U{\cdots}
	\ar@{-} (15,-17.5); (140, -17.5);
	\ar@{-} (15,-17.5); (15, -34.5);
	\ar@{-} (15,-34.5); (140, -34.5);
	\ar@{-} (140,-17.5); (140, -34.5)
	\ar@{} (0,0);(78,-26) *{\text{$K_{2b}$}};
\end{xy}



  \caption{The graph $G_{1, b, c, 2(c - b)}^{(2)}$}
  \label{fig:G_{1, b, c, 2(c - b)}^{(2)}}
\end{figure}

Then ${\rm ind}$-${\rm match}(G_{1, b, c, 2(c - b)}^{(2)}) = 1$ since each edge of $G_{1, b, c, 2(c - b)}^{(2)}$ contains 
a vertex belonging to $V_{2b}$.
Since $\bigcup_{i = 1}^{b} \{v_{i}, v_{b + i}\} \subset E(G_{1, b, c, 2(c - b)}^{(2)})$ is a maximal matching, one has 
${\rm min}$-${\rm match}(G_{1, b, c, 2(c - b)}^{(2)}) \leq b$. 
Moreover, by virtue of Lemma \ref{starK2n} and \ref{inducedsubgraph}, it follows that  $b = {\rm min}$-${\rm match}(K_{2b}) \leq {\rm min}$-${\rm match}(G_{1, b, c, 2(c - b)}^{(2)})$. 
Therefore ${\rm min}$-${\rm match}(G_{1, b, c, 2(c - b)}^{(2)}) = b$. 
Let $E^{(2)} = \bigcup_{i = 1}^{c - b} \left\{ \{v_{i}, x_{i}\}, \{v_{b + j}, x_{c - b + j}\} \right\}  \cup \bigcup_{j = 1}^{2b - c} \{v_{c - b + j }, v_{c + j}\}$. 
Then $E^{(2)}$ is a perfect matching of $G_{1, b, c, 2(c - b)}^{(2)}$ 
(i.e. a matching in which every vertex of $G_{1, b, c, 2(c - b)}^{(2)}$ is contained in some edge belonging to $E^{(2)}$). 
Hence ${\rm match}(G_{1, b, c, 2(c - b)}^{(2)}) = c$. 
A subset $\{x_{1}, x_{2}, \ldots, x_{2(c - b)}\} \subset V(G_{1, b, c, 2(c - b)}^{(2)})$ 
is an independent set of $G_{1, b, c, 2(c - b)}^{(2)}$. 
Hence we have $\dim (G_{1, b, c, 2(c - b)}^{(2)}) \geq 2(c - b)$ by Lemma \ref{dim(G)}. 
Let $W \subset V(G_{1, b, c, 2(c - b)}^{(2)})$ with $|W| \geq 2(c - b) + 1$. 
Then $|W \cap V_{2b}| \geq 1$. 
We may assume $v_{1} \in W$. 
Assume that $W$ is an independent set. 
Since $\deg(v_{1}) = b + c - 1$ and $b < c$, we have $|W| \leq |V(G_{1, b, c, 2(c - b)}^{(2)})| - \deg (v_{1}) = c - b + 1 < 2(c - b) + 1$, 
but this is a contradiction. 
Thus $W$ is not an independent set. 
Therefore one has $\dim (G_{1, b, c, 2(c - b)}^{(2)}) = 2(c - b)$. 

$\bullet \ ${\bf Case 3: $a = 1, 1 \leq b < c \leq 2b$ and $d > 2(c - b). $} \\ 
Let $G_{1, b, c, d}^{(3)}$ be the graph such that 
\begin{eqnarray*}
V(G_{1, b, c, d}^{(3)}) &=& V_{2b} \cup \{x_{1}, x_{2}, \ldots, x_{2(c - b)}\} \cup \{ y_{1}, y_{2}, \ldots, y_{d - 2(c - b)} \}, \\
E(G_{1, b, c, d}^{(3)}) &=& \left\{ \{v_{i}, v_{j}\} : 1 \leq i < j \leq 2b \right\} \cup \bigcup_{k = 1}^{b} \bigcup_{\ell = 1}^{c - b} \left\{ \{v_{k}, x_{\ell}\}, \{ v_{b + k}, x_{c - b + \ell} \} \right\} \\
&\cup& \left\{ \{v_{1}, y_{k}\} : 1 \leq k \leq d - 2(c - b) \right\};
\end{eqnarray*}
see Figure \ref{fig:G_{1, b, c, d}^{(3)}}. 

\begin{figure}[htbp]
\centering
\bigskip

\begin{xy}
	\ar@{} (0,0);(20, -16)  *++! U{v_{1}} *\cir<4pt>{} = "V1"
	\ar@{} (0,0);(35, -16)  *++! U{v_{2}} *\cir<4pt>{} = "V2"
	\ar@{} (0,0);(70, -16)  *++! U{v_{b}} *\cir<4pt>{} = "Vb"
	\ar@{-} "V1";(25, -4) *++! D{x_{1}} *\cir<4pt>{} = "X1";
	\ar@{-} "V1";(35, -4) *++!D{x_{2}} *\cir<4pt>{} = "X2";
	\ar@{} "V1"; (49, 0) *++!U{\cdots}
	\ar@{} "V1"; (49, -12) *++!U{\cdots}
	\ar@{-} "V1";(65, -4) *++!D{x_{c - b}} *\cir<4pt>{} = "Xc-b";
	\ar@{-} "V2"; "X1";
	\ar@{-} "V2"; "X2";
	\ar@{-} "V2"; "Xc-b";
	\ar@{-} "Vb"; "X1";
	\ar@{-} "Vb"; "X2";
	\ar@{-} "Vb"; "Xc-b";
	\ar@{} (0,0);(85, -16)  *++! U{v_{b + 1}} *\cir<4pt>{} = "Vb+1";
	\ar@{} (0,0);(100, -16)  *++! U{v_{b + 2}} *\cir<4pt>{} = "Vb+2";
	\ar@{} (0,0);(135, -16)  *++! U{v_{2b}} *\cir<4pt>{} = "V2b";
	\ar@{-} "Vb+1";(90, -4) *++! R{x_{c - b + 1}} *\cir<4pt>{} = "Xc-b+1";
	\ar@{-} "Vb+1";(100, -4) *++!D{x_{c - b + 2}} *\cir<4pt>{} = "Xc-b+2";
	\ar@{-} "Vb+1";(130, -4) *++!D{x_{2(c - b)}} *\cir<4pt>{} =  "X2(c-b)";
	\ar@{-} "Vb+2"; "Xc-b+1";
	\ar@{-} "Vb+2"; "Xc-b+2";
	\ar@{-} "Vb+2"; "X2(c-b)";
	\ar@{-} "V2b"; "Xc-b+1";
	\ar@{-} "V2b"; "Xc-b+2";
	\ar@{-} "V2b"; "X2(c-b)";
	\ar@{-} "V1";(5, 4) *++! R{y_{1}} *\cir<4pt>{};
	\ar@{-} "V1";(5, -2) *++! R{y_{2}} *\cir<4pt>{};
	\ar@{} (0,0); (5, -2.5) *++!U{\vdots}
	\ar@{-} "V1";(5, -16) *++! U{y_{d - 2(c - b)}} *\cir<4pt>{};
	\ar@{} "V1"; (115, 0) *++!U{\cdots}
	\ar@{} "V1"; (115, -12) *++!U{\cdots}
	\ar@{-} (15,-17.5); (140, -17.5);
	\ar@{-} (15,-17.5); (15, -34.5);
	\ar@{-} (15,-34.5); (140, -34.5);
	\ar@{-} (140,-17.5); (140, -34.5)
	\ar@{} (0,0);(78,-26) *{\text{$K_{2b}$}};
\end{xy}

\bigskip

  \caption{The graph $G_{1, b, c, d}^{(3)}$}
  \label{fig:G_{1, b, c, d}^{(3)}}
\end{figure}
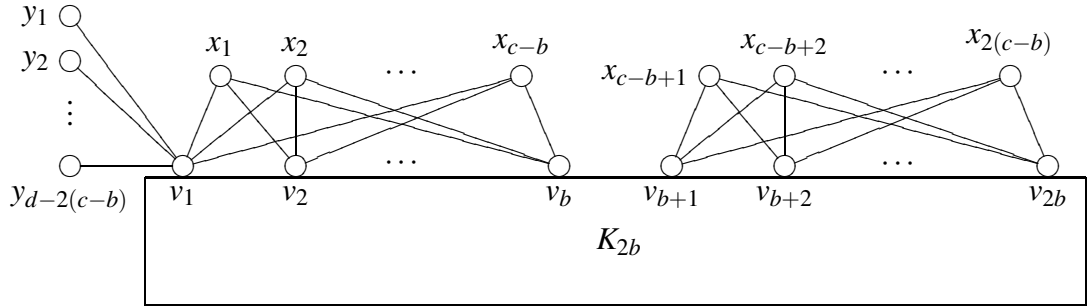

Then ${\rm ind}$-${\rm match}(G_{1, b, c, d}^{(3)}) = 1$ since each edge of $G_{1, b, c, d}^{(3)}$ contains 
a vertex belonging to $V_{2b}$.
By the same argument in the Case 2, we can see that ${\rm min}$-${\rm match}(G_{1, b, c, d}^{(3)}) = b$. 
Since $G_{1, b, c, 2(c - b)}^{(2)}$ which appears in Case 2 is an induced subgraph of $G_{1, b, c, d}^{(3)}$, 
hence ${\rm match}(G_{1, b, c, d}^{(3)}) \geq c$ from Lemma \ref{inducedsubgraph}. 
Moreover, $G_{1, b, c, 2(c - b) + 1}^{(3)}$ is also an induced subgraph of $G_{1, b, c, d}^{(3)}$ 
with $|V(G_{1, b, c, 2(c - b) + 1}^{(3)})| = 2c + 1$. 
Thus, by virtue of Lemma \ref{inducedsubgraph} and \ref{match}, we have 
${\rm match}(G_{1, b, c, d}^{(3)}) = {\rm match}(G_{1, b, c, 2(c - b) + 1}^{(3)}) \leq c$. 
Therefore ${\rm match}(G_{1, b, c, d}^{(3)}) = c$.  
A subset $\{x_{1}, x_{2}, \ldots, x_{2(c - b)}\} \cup \{ y_{1}, y_{2}, \ldots, y_{d - 2(c - b)} \} \subset V(G_{1, b, c, d}^{(3)})$ 
is an independent set of $G_{1, b, c, d}^{(3)}$. 
Hence $\dim (G_{1, b, c, d}^{(3)}) \geq 2(c - b) + d - 2(c - b) = d$ by Lemma \ref{dim(G)}. 
Let $W \subset V(G_{1, b, c, d}^{(3)})$ with $|W| \geq d + 1$. 
Then there exists a vertex $v \in V_{2b}$. 
Assume that $W$ is an independent set.  
Since $\deg(v) \geq b + c - 1$ and $b < c$, we have 
$|W| \leq |V(G_{1, b, c, d)}^{(3)})| - \deg (v) \leq d + b - c + 1 < d + 1$, but this is a contradiction. 
Thus $W$ is not an independent set. 
Therefore $\dim (G_{1, b, c, d}^{(3)}) = d$.

$\bullet \ ${\bf Case 4: $1 < a \leq b = c$ and $d = a. $} \\ 
Let $G_{a, b, b, a}^{(4)}$ be the graph such that 
\begin{eqnarray*}
V(G_{a, b, b, a}^{(4)}) &=& V_{2b} \cup \{x\}, \\
E(G_{a, b, b, a}^{(4)}) &=& \bigcup_{i = 1}^{a - 1} \{v_{2i - 1}, v_{2i}\} \cup \left\{ \{v_{j}, v_{k}\} : 2a - 1 \leq j < k \leq 2b \right\} \\
&\cup& \left\{ \{v_{\ell}, x\} : 1 \leq \ell \leq 2b \right\};
\end{eqnarray*}
see Figure \ref{fig:G_{a, b, b, a}^{(4)}}. 

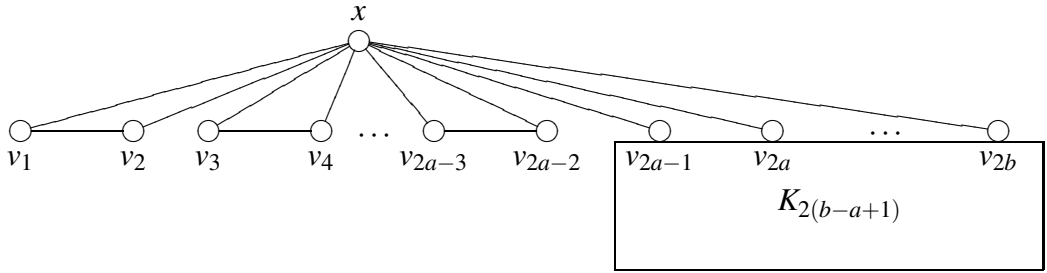
\begin{figure}[htbp]
\centering
\bigskip

\begin{xy}
	\ar@{} (0,0);(15, -16)  *++! U{v_{1}} *\cir<4pt>{} = "V1";
	\ar@{} (0,0);(40, -16)  *++! U{v_{3}} *\cir<4pt>{} = "V3";
	\ar@{} (0,0); (62, -13) *++!U{\cdots}
	\ar@{} (0,0);(70, -16)  *++! U{v_{2a-3}} *\cir<4pt>{} = "V2a-3";
	\ar@{-} "V1"; (30, -16) *++! U{v_{2}} *\cir<4pt>{} = "V2";
	\ar@{-} "V3"; (55, -16) *++! U{v_{4}} *\cir<4pt>{} = "V4";
	\ar@{-} "V2a-3"; (85, -16) *++! U{v_{2a - 2}} *\cir<4pt>{} = "V2(a-1)";
	\ar@{} "V1"; (60, -4) *++! D{x} *\cir<4pt>{} = "X";
	\ar@{-} "X"; "V1";
	\ar@{-} "X"; "V3";
	\ar@{-} "X"; "V2a-3";
	\ar@{-} "X"; "V2";
	\ar@{-} "X"; "V4";
	\ar@{-} "X"; "V2(a-1)";
	\ar@{-} "X";(100, -16)  *++! U{v_{2a - 1}} *\cir<4pt>{};
	\ar@{-} "X";(115, -16)  *++! U{v_{2a}} *\cir<4pt>{};
	\ar@{} "X"; (130, -12.8) *++!U{\cdots}
	\ar@{-} "X";(145, -16)  *++! U{v_{2b}} *\cir<4pt>{};
	\ar@{-} (94,-17.5); (151, -17.5);
	\ar@{-} (94,-17.5); (94, -34.5);
	\ar@{-} (94,-34.5); (151, -34.5);
	\ar@{-} (151,-17.5); (151, -34.5)
	\ar@{} (0,0);(124,-26) *{\text{$K_{2(b - a + 1)}$}};
\end{xy}

\bigskip

  \caption{The graph $G_{a, b, b, a}^{(4)}$}
  \label{fig:G_{a, b, b, a}^{(4)}}
\end{figure}

Let $H^{(4)} = (G_{a, b, b, a}^{(4)})_{V_{2b}}$ be the induced subgraph of $G_{a, b, b, a}^{(4)}$ on $V_{2b}$. 
Then we can regard that $G_{a, b, b, a}^{(4)}$ is the $\emptyset$-suspension of $H^{(4)}$ and 
$H^{(4)}$ is the disjoint union of $(a - 1)$ $K_{2}$ and $K_{2(b - a + 1)}$. 
By virtue of Lemma \ref{starK2n} and \ref{disconnected}, it follows that 
${\rm ind}$-${\rm match}(G_{a, b, b, a}^{(4)}) = \dim (G_{a, b, b, a}^{(4)}) = a$. 
Moreover, it is easy to see that ${\rm min}$-${\rm match}(G_{a, b, b, a}^{(4)}) = {\rm match}(G_{a, b, b, a}^{(4)}) = a - 1 + b - a + 1 = b$. \\

$\bullet \ ${\bf Case 5: $1 < a \leq b = c$ and $d > a. $} \\ 
Let $G_{a, b, b, d}^{(5)}$ be the graph such that 
\begin{eqnarray*}
V(G_{a, b, b, d}^{(5)}) &=& V_{2b} \cup \{x\} \cup \{y_{1}, y_{2}, \ldots, y_{d - a - 1}\}, \\
E(G_{a, b, b, d}^{(5)}) &=& \bigcup_{i = 1}^{a - 1} \{v_{2i - 1}, v_{2i}\} \cup \left\{ \{v_{j}, v_{k}\} : 2a - 1 \leq j < k \leq 2b \right\} \\
&\cup& \left\{ \{v_{2\ell - 1}, x\} : 1 \leq \ell \leq a \right\} \cup \left\{ \{v_{1}, y_{m}\} : 1 \leq m \leq d - a - 1 \right\};
\end{eqnarray*}
see Figure \ref{fig:G_{a, b, b, d}^{(5)}}. 

\begin{figure}[htbp]
\centering
\bigskip

\begin{xy}
	\ar@{} (0,0);(15, -16)  *++! R{v_{1}} *\cir<4pt>{} = "V1";
	\ar@{} (0,0);(40, -16)  *++! U{v_{3}} *\cir<4pt>{} = "V3";
	\ar@{} (0,0); (62, -13) *++!U{\cdots}
	\ar@{} (0,0);(70, -16)  *++! U{v_{2a-3}} *\cir<4pt>{} = "V2a-3";
	\ar@{-} "V1"; (30, -16) *++! U{v_{2}} *\cir<4pt>{} = "V2";
	\ar@{-} "V3"; (55, -16) *++! U{v_{4}} *\cir<4pt>{} = "V4";
	\ar@{-} "V2a-3"; (85, -16) *++! U{v_{2a - 2}} *\cir<4pt>{} = "V2(a-1)";
	\ar@{} "V1"; (60, -4) *++! D{x} *\cir<4pt>{} = "X";
	\ar@{-} "X"; "V1";
	\ar@{-} "X"; "V3";
	\ar@{-} "X"; "V2a-3";
	\ar@{-} "X";(100, -16)  *++! U{v_{2a - 1}} *\cir<4pt>{};
	\ar@{} "X";(115, -16)  *++! U{v_{2a}} *\cir<4pt>{};
	\ar@{} "X"; (130, -12.8) *++!U{\cdots}
	\ar@{} "X";(145, -16)  *++! U{v_{2b}} *\cir<4pt>{};
	\ar@{-} "V1";(0, -32) *++! U{y_{1}} *\cir<4pt>{};
	\ar@{-} "V1";(10, -32) *++! U{y_{2}} *\cir<4pt>{};
	\ar@{} "V1";(20, -28.4) *++!U{\cdots};
	\ar@{-} "V1";(30, -32) *++! U{y_{d - a - 1}} *\cir<4pt>{};
	\ar@{-} (94,-17.5); (151, -17.5);
	\ar@{-} (94,-17.5); (94, -34.5);
	\ar@{-} (94,-34.5); (151, -34.5);
	\ar@{-} (151,-17.5); (151, -34.5)
	\ar@{} (0,0);(124,-26) *{\text{$K_{2(b - a + 1)}$}};
\end{xy}


  \caption{The graph $G_{a, b, b, d}^{(5)}$}
  \label{fig:G_{a, b, b, d}^{(5)}}
\end{figure}
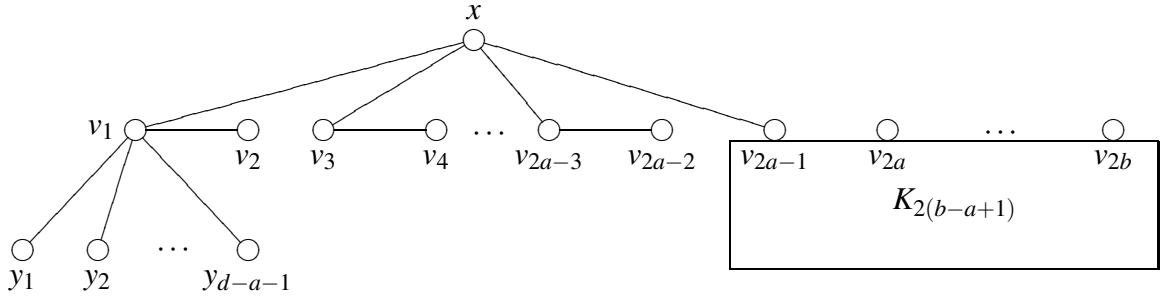

Let $H^{(5)} = (G_{a, b, b, d}^{(5)})_{V_{2b} \cup \{x\}}$ be the induced subgraph of $G_{a, b, b, d}^{(5)}$ on $V_{2b} \cup \{x\}$. 
It is easy to see that ${\rm ind}$-${\rm match}(H^{(5)}) = a$ and 
${\rm min}$-${\rm match}(H^{(5)}) = {\rm match}(H^{(5)}) = a - 1 + b - a + 1 = b$. 
Hence we have ${\rm ind}$-${\rm match}(G_{a, b, b, d}^{(5)}) = a$ and 
${\rm min}$-${\rm match}(G_{a, b, b, d}^{(5)}) = {\rm match}(G_{a, b, b, d}^{(5)}) = b$ by virtue of Lemma \ref{match}. 
A subset $\{v_{2}, v_{4}, \ldots, v_{2a - 2}, v_{2b}\} \cup \{x\} \cup \{ y_{1}, y_{2}, \ldots, y_{d - a - 1} \} \subset V(G_{a, b, b, d}^{(5)})$ 
is an independent set of $G_{a, b, b, d}^{(5)}$. 
Hence $\dim (G_{a, b, b, d}^{(5)}) \geq a + 1 + d - a - 1 = d$ by Lemma \ref{dim(G)}. 
Let $W \subset V(G_{a, b, b, d}^{(5)})$ with $|W| \geq d + 1$. 
Then $|W \cap V_{2b}| \geq a + 1$. 
Thus $W$ is not an independent set. 
Therefore $\dim (G_{a, b, b, d}^{(5)}) = d$.

$\bullet \ ${\bf Case 6: $1 < a \leq b < c \leq 2b$ and $d \geq 2(c - b) \geq a. $} \\ 
Let $Y_{d - 2(c - b)} = \{y_{1}, y_{2}, \ldots, y_{d - 2(c - b)}\}$.
Let $G_{a, b, c, d}^{(6)}$ be the graph such that 
\begin{eqnarray*}
V(G_{a, b, c, d}^{(6)}) &=& V_{2b} \cup \{x_{1}, x_{2}, \ldots, x_{2(c - b)}\} \cup Y_{d - 2(c - b)}, \\
E(G_{a, b, c, d}^{(6)}) &=& \bigcup_{i = 1}^{a - 1} \{v_{2i - 1}, v_{2i}\} \cup \left\{ \{v_{j}, v_{k}\} : 2a - 1 \leq j < k \leq 2b \right\} \\
&\cup& \left\{ \{v_{\ell}, x_{m}\} : 1 \leq \ell \leq 2b, 1 \leq m \leq 2(c - b) \right\} \\
&\cup& \left\{ \{v_{1}, y_{p}\} : 1 \leq p \leq d - 2(c - b) \right\};
\end{eqnarray*}
see Figure \ref{fig:G_{a, b, c, d}^{(6)}}. 

\begin{figure}[htbp]
\centering
\bigskip

\begin{xy}
	\ar@{} (0,0);(15, -16)  *++! R{v_{1}} *\cir<4pt>{} = "V1";
	\ar@{} (0,0);(40, -16)  *++! U{v_{3}} *\cir<4pt>{} = "V3";
	\ar@{} (0,0); (62, -13) *++!U{\cdots}
	\ar@{} (0,0);(70, -16)  *++! U{v_{2a-3}} *\cir<4pt>{} = "V2a-3";
	\ar@{-} "V1"; (30, -16) *++! U{v_{2}} *\cir<4pt>{} = "V2";
	\ar@{-} "V3"; (55, -16) *++! U{v_{4}} *\cir<4pt>{} = "V4";
	\ar@{-} "V2a-3"; (85, -16) *++! U{v_{2a - 2}} *\cir<4pt>{} = "V2a-2";
	\ar@{} "V1"; (50, 0) *++! D{x_{1}} *\cir<4pt>{} = "X1";
	\ar@{} "V1"; (70, 0) *++! D{x_{2}} *\cir<4pt>{} = "X2";
	\ar@{} "V1"; (78, 0) *++!L{\cdots}
	\ar@{} "V1"; (95, 0) *++! D{x_{2(c - b)}} *\cir<4pt>{} = "X2(c-b)";
	\ar@{-} "X1"; "V1";
	\ar@{-} "X1"; "V2";
	\ar@{-} "X1"; "V3";
	\ar@{-} "X1"; "V4";
	\ar@{-} "X1"; "V2a-3";
	\ar@{-} "X1"; "V2a-2";
	\ar@{-} "X1";(100, -16)  *++! U{v_{2a - 1}} *\cir<4pt>{} = "V2a-1";
	\ar@{-} "X1";(115, -16)  *++! U{v_{2a}} *\cir<4pt>{} = "V2a";
	\ar@{} "X1"; (130, -12.8) *++!U{\cdots}
	\ar@{-} "X1";(145, -16)  *++! U{v_{2b}} *\cir<4pt>{} = "V2b";
	\ar@{-} "X2"; "V1";
	\ar@{-} "X2"; "V2";
	\ar@{-} "X2"; "V3";
	\ar@{-} "X2"; "V4";
	\ar@{-} "X2"; "V2a-3";
	\ar@{-} "X2"; "V2a-2";
	\ar@{-} "X2"; "V2a-1";
	\ar@{-} "X2"; "V2a";
	\ar@{-} "X2"; "V2b";
	\ar@{-} "X2(c-b)"; "V1";
	\ar@{-} "X2(c-b)"; "V2";
	\ar@{-} "X2(c-b)"; "V3";
	\ar@{-} "X2(c-b)"; "V4";
	\ar@{-} "X2(c-b)"; "V2a-3";
	\ar@{-} "X2(c-b)"; "V2a-2";
	\ar@{-} "X2(c-b)"; "V2a-1";
	\ar@{-} "X2(c-b)"; "V2a";
	\ar@{-} "X2(c-b)"; "V2b";
	\ar@{-} "V1";(0, -32) *++! U{y_{1}} *\cir<4pt>{};
	\ar@{-} "V1";(10, -32) *++! U{y_{2}} *\cir<4pt>{};
	\ar@{} "V1";(20, -28.4) *++!U{\cdots};
	\ar@{-} "V1";(30, -32) *++! U{y_{d - 2(c - b)}} *\cir<4pt>{};
	\ar@{-} (94,-17.5); (151, -17.5);
	\ar@{-} (94,-17.5); (94, -34.5);
	\ar@{-} (94,-34.5); (151, -34.5);
	\ar@{-} (151,-17.5); (151, -34.5)
	\ar@{} (0,0);(124,-26) *{\text{$K_{2(b - a + 1)}$}};
\end{xy}


  \caption{The graph $G_{a, b, c, d}^{(6)}$}
  \label{fig:G_{a, b, c, d}^{(6)}}
\end{figure}
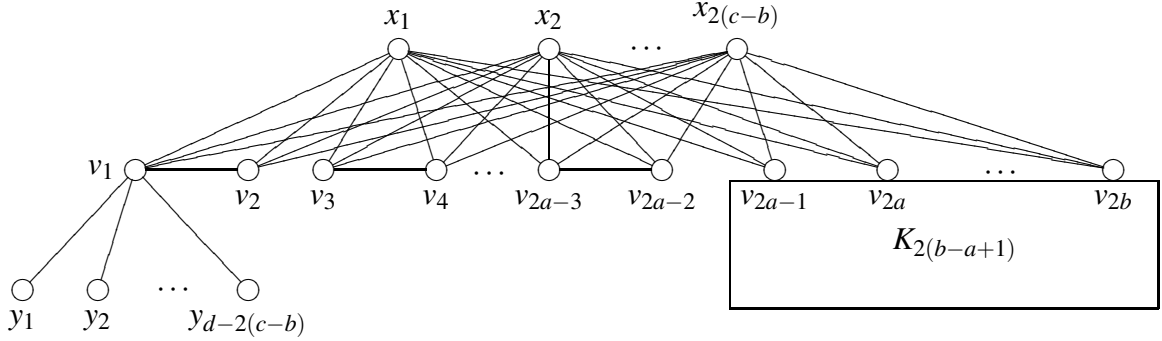

Let $S_{0} = Y_{d - 2(c - b)}$. 
For $1 \leq i \leq 2(c - b) - 1$, we define $S_{i} = Y_{d - 2(c - b)} \cup \{x_{1}, \ldots, x_{i}\}$. 
Then $S_{i}$ are independent sets and $|S_{i}| = d - 2(c - b) + i$ for $i = 0, 1, \ldots, 2(c - b) - 1$. 
Let $H_{0}^{(6)} = (G_{a, b, c, d}^{(6)})_{V_{2b} \cup Y_{d - 2(c - b)}}$ be the induced subgraph on $V_{2b} \cup Y_{d - 2(c - b)}$. 
We can regard that $H_{0}^{(6)}$ is the disjoint union of the star graph $G^{{\rm star}(v_{1})}_{d - 2(c - b) + 1}$, $(a - 2)K_{2}$ and 
$K_{2(b - a + 1)}$. 
Hence we have ${\rm ind}$-${\rm match}(H_{0}^{(6)}) = a$ and $\dim (H_{0}^{(6)}) = d - 2(c - b) + 1 + a - 2 + 1 = d - 2(c - b) + a$ 
by Lemma \ref{starK2n}. 

For $1 \leq j \leq 2(c - b)$, we define $H_{j}^{(6)} = (H_{j - 1}^{(6)})^{S_{j - 1}}$ inductively. 
Then, Lemma \ref{S-suspension} says that ${\rm ind}$-${\rm match}(H_{j}^{(6)}) = a$ for all $j = 0, 1, \ldots, 2(c - b)$ and 
\begin{center}
$\dim (H_{j}^{(6)}) = \begin{cases} \displaystyle d - 2(c - b) + a & \text{$(0 \leq j \leq a)$}, \\ \displaystyle d - 2(c - b) + j & \text{($a + 1 \leq j \leq 2(c - b))$}. \end{cases}$ 
\end{center}
Since $G_{a, b, c, d}^{(6)} = H_{2(c - b)}^{(6)}$, one has ${\rm ind}$-${\rm match}(G_{a, b, c, d}^{(6)}) = a$ and 
$\dim (G_{a, b, c, d}^{(6)}) = d$ from Lemma \ref{S-suspension}. 

Note that $\bigcup_{i = 1}^{b} \{v_{2i - 1}, v_{2i}\} \subset E(G_{a, b, c, d}^{(6)})$ is a maximal matching of $G_{a, b, c, d}^{(6)}$. 
Hence one has ${\rm min}$-${\rm match}(G_{a, b, c, d}^{(6)}) \leq b$. 
Moreover, by virtue of Lemma \ref{starK2n} and \ref{inducedsubgraph}, it follows that  $b = {\rm min}$-${\rm match}((G_{a, b, c, d}^{(6)})_{V_{2b}}) \leq {\rm min}$-${\rm match}(G_{a, b, c, d}^{(6)})$. 
Therefore ${\rm min}$-${\rm match}(G_{a, b, c, d}^{(6)})) = b$.
Since $\bigcup_{i = 1}^{2(c - b)} \{v_{i}, x_{i}\} \cup \bigcup_{j = 1}^{2b - c} \{v_{2(c - b) + 2j - 1}, v_{2(c - b) + 2j}\} $ is a matching, 
hence we have ${\rm match}(G_{a, b, c, d}^{(6)}) \geq 2(c - b) + 2b - c = c$. 
In addition, since $|V(G_{a, b, c, 2(c - b) + 1}^{(6)}))| = 2c + 1$, we also have 
${\rm match}(G_{a, b, c, d}^{(6)}) = {\rm match}(G_{a, b, c, 2(c - b) + 1}^{(6)}) \leq c$ by Lemma \ref{match}. 
Thus one has ${\rm match}(G_{a, b, c, d}^{(6)}) = c$. 

$\bullet \ ${\bf Case 7: $1 < a \leq b < c \leq 2b$ and $d \geq a > 2(c - b). $} \\ 
Let $Y_{d - a} = \{y_{1}, y_{2}, \ldots, y_{d - a}\}$.
Let $G_{a, b, c, d}^{(7)}$ be the graph such that 
\begin{eqnarray*}
V(G_{a, b, c, d}^{(7)}) &=& V_{2b} \cup \{x_{1}, x_{2}, \ldots, x_{2(c - b)}\} \cup Y_{d - a}, \\
E(G_{a, b, c, d}^{(7)}) &=& \bigcup_{i = 1}^{a - 1} \{v_{2i - 1}, v_{2i}\} \cup \left\{ \{v_{j}, v_{k}\} : 2a - 1 \leq j < k \leq 2b \right\} \\
&\cup& \left\{ \{v_{\ell}, x_{m}\} : 1 \leq \ell \leq 2b, 1 \leq m \leq 2(c - b) \right\} \\
&\cup& \left\{ \{v_{1}, y_{p}\} : 1 \leq p \leq d - a \right\};
\end{eqnarray*}
see Figure \ref{fig:G_{a, b, c, d}^{(7)}}. 

\begin{figure}[htbp]
\centering
\bigskip

\begin{xy}
	\ar@{} (0,0);(15, -16)  *++! R{v_{1}} *\cir<4pt>{} = "V1";
	\ar@{} (0,0);(40, -16)  *++! U{v_{3}} *\cir<4pt>{} = "V3";
	\ar@{} (0,0); (62, -13) *++!U{\cdots}
	\ar@{} (0,0);(70, -16)  *++! U{v_{2a-3}} *\cir<4pt>{} = "V2a-3";
	\ar@{-} "V1"; (30, -16) *++! U{v_{2}} *\cir<4pt>{} = "V2";
	\ar@{-} "V3"; (55, -16) *++! U{v_{4}} *\cir<4pt>{} = "V4";
	\ar@{-} "V2a-3"; (85, -16) *++! U{v_{2a - 2}} *\cir<4pt>{} = "V2a-2";
	\ar@{} "V1"; (50, 0) *++! D{x_{1}} *\cir<4pt>{} = "X1";
	\ar@{} "V1"; (70, 0) *++! D{x_{2}} *\cir<4pt>{} = "X2";
	\ar@{} "V1"; (78, 0) *++!L{\cdots}
	\ar@{} "V1"; (95, 0) *++! D{x_{2(c - b)}} *\cir<4pt>{} = "X2(c-b)";
	\ar@{-} "X1"; "V1";
	\ar@{-} "X1"; "V2";
	\ar@{-} "X1"; "V3";
	\ar@{-} "X1"; "V4";
	\ar@{-} "X1"; "V2a-3";
	\ar@{-} "X1"; "V2a-2";
	\ar@{-} "X1";(100, -16)  *++! U{v_{2a - 1}} *\cir<4pt>{} = "V2a-1";
	\ar@{-} "X1";(115, -16)  *++! U{v_{2a}} *\cir<4pt>{} = "V2a";
	\ar@{} "X1"; (130, -12.8) *++!U{\cdots}
	\ar@{-} "X1";(145, -16)  *++! U{v_{2b}} *\cir<4pt>{} = "V2b";
	\ar@{-} "X2"; "V1";
	\ar@{-} "X2"; "V2";
	\ar@{-} "X2"; "V3";
	\ar@{-} "X2"; "V4";
	\ar@{-} "X2"; "V2a-3";
	\ar@{-} "X2"; "V2a-2";
	\ar@{-} "X2"; "V2a-1";
	\ar@{-} "X2"; "V2a";
	\ar@{-} "X2"; "V2b";
	\ar@{-} "X2(c-b)"; "V1";
	\ar@{-} "X2(c-b)"; "V2";
	\ar@{-} "X2(c-b)"; "V3";
	\ar@{-} "X2(c-b)"; "V4";
	\ar@{-} "X2(c-b)"; "V2a-3";
	\ar@{-} "X2(c-b)"; "V2a-2";
	\ar@{-} "X2(c-b)"; "V2a-1";
	\ar@{-} "X2(c-b)"; "V2a";
	\ar@{-} "X2(c-b)"; "V2b";
	\ar@{-} "V1";(0, -32) *++! U{y_{1}} *\cir<4pt>{};
	\ar@{-} "V1";(10, -32) *++! U{y_{2}} *\cir<4pt>{};
	\ar@{} "V1";(20, -28.4) *++!U{\cdots};
	\ar@{-} "V1";(30, -32) *++! U{y_{d - a}} *\cir<4pt>{};
	\ar@{-} (94,-17.5); (151, -17.5);
	\ar@{-} (94,-17.5); (94, -34.5);
	\ar@{-} (94,-34.5); (151, -34.5);
	\ar@{-} (151,-17.5); (151, -34.5)
	\ar@{} (0,0);(124,-26) *{\text{$K_{2(b - a + 1)}$}};
\end{xy}

\bigskip

  \caption{The graph $G_{a, b, c, d}^{(7)}$}
  \label{fig:G_{a, b, c, d}^{(7)}}
\end{figure}
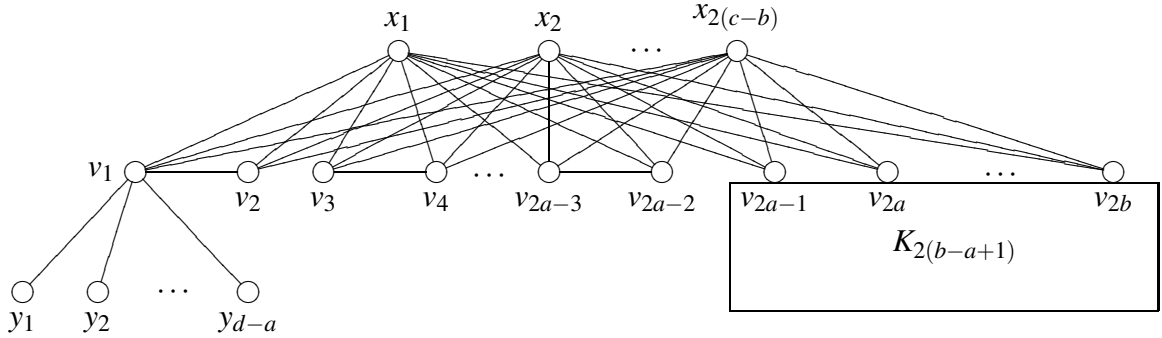

Applying the same argument as in the Case 6, we have that 
${\rm min}$-${\rm match}(G_{a, b, c, d}^{(7)}) = b$ and ${\rm match}(G_{a, b, c, d}^{(7)}) = c$. 

Let $S_{0} = Y_{d - a}$. 
For $1 \leq i \leq 2(c - b) - 1$, we define $S_{i} = Y_{d - a} \cup \{x_{1}, \ldots, x_{i}\}$. 
Then $S_{i}$ are independent sets and $|S_{i}| = d - a + i$ for $i = 0, 1, \ldots, 2(c - b) - 1$. 
In particular, $|S_{i}| < d$ since $a > 2(c - b)$. 
Let $H_{0}^{(7)} = (G_{a, b, c, d}^{(7)})_{V_{2b} \cup Y_{d - a}}$ be the induced subgraph on $V_{2b} \cup Y_{d - a}$. 
We can regard that $H_{0}^{(7)}$ is the disjoint union of the star graph $G^{{\rm star}(v_{1})}_{d - a + 1}$, $(a - 2)K_{2}$ and 
$K_{2(b - a + 1)}$. 
Hence ${\rm ind}$-${\rm match}(H_{0}^{(7)}) = a$ and $\dim (H_{0}^{(7)}) = d - a + 1 + a - 2 + 1 = d$ 
by Lemma \ref{starK2n}. 

For $1 \leq j \leq 2(c - b)$, we define $H_{j}^{(7)} = (H_{j - 1}^{(7)})^{S_{j - 1}}$ inductively. 
Then, Lemma \ref{S-suspension} says that ${\rm ind}$-${\rm match}(H_{j}^{(7)}) = a$ and $\dim (H_{j}^{(7)}) = d$ 
for all $j = 0, 1, \ldots, 2(c - b)$.   
Since $G_{a, b, c, d}^{(7)} = H_{2(c - b)}^{(7)}$, it follows that ${\rm ind}$-${\rm match}(G_{a, b, c, d}^{(7)}) = a$ and 
$\dim (G_{a, b, c, d}^{(7)}) = d$. 
\end{proof}




\begin{Remark}\normalfont
Let $a, b, m, n$ be non-negative integers with $m \leq n$ and $1 \leq n$. 
Let $G_{a, b, m, n}$ be the graph which appears in \cite[p. 176]{HHKT}. 
\cite[Theorem 2.3]{HHKT} says that ${\rm ind}$-${\rm match}(G_{a, b, m, n}) = a + b + 1$, 
${\rm min}$-${\rm match}(G_{a, b, m, n}) = a + b + n$ and ${\rm match}(G_{a, b, m, n}) = 2a + b + n + m$. 
Note that
\begin{center}
$\dim (G_{a, b, m, n}) = \begin{cases} \displaystyle 3a + 2b + 2m & \text{$(m = n)$}, \\ \displaystyle 3a + 2b + 2m + 1 & \text{$(m < n)$}. \end{cases}$ 
\end{center}
In particular, we have
\begin{itemize}
	\item Suppose that ${\rm ind}$-${\rm match}(G_{a, b, m, n}) = 1$. 
	Then $a = b = 0$. Hence
	\begin{center}
	$\dim (G_{a, b, m, n}) = \begin{cases} \displaystyle 2\{{\rm match}(G_{a, b, m, n}) - {\rm min}$-${\rm match}(G_{a, b, m, n})\} & \text{$(m = n)$}, \\ \displaystyle 2\{{\rm match}(G_{a, b, m, n}) - {\rm min}$-${\rm match}(G_{a, b, m, n})\} + 1 & \text{$	(m < n)$}. \end{cases}$ 
	\end{center}
	\item Suppose that ${\rm min}$-${\rm match}(G_{a, b, m, n}) = {\rm match}(G_{a, b, m, n})$. 
	Then $a = m = 0$. Hence $\dim (G_{a, b, m, n}) = 2{\rm ind}$-${\rm match}(G_{a, b, m, n}) - 1$. 
	\item Suppose that ${\rm match}(G_{a, b, m, n}) = 2{\rm min}$-${\rm match}(G_{a, b, m, n})$. 
	Then $b = 0$ and $m = n$. 
	Hence $\dim (G_{a, b, m, n}) = 2{\rm min}$-${\rm match}(G_{a, b, m, n}) + {\rm ind}$-${\rm match}(G_{a, b, m, n}) - 1$. 
	\item Suppose that ${\rm min}$-${\rm match}(G_{a, b, m, n}) < {\rm match}(G_{a, b, m, n}) < 2{\rm min}$-${\rm match}(G_{a, b, m, n})$.
	Then $0 < a + m$ and $m < b + n$. 
	Hence $\dim (G_{a, b, m, n}) \geq 2{\rm ind}$-${\rm match}(G_{a, b, m, n})$.
\end{itemize}
\end{Remark} 

\begin{Example}\normalfont
\begin{enumerate}
	\item Let $a = 1, b = c = 2$ and $d = 3$. Then \\
	$\ $ \\
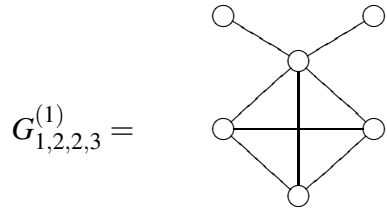
	\begin{xy}
		\ar@{} (0,0);(45,-15) *{\text{$G_{1, 2, 2, 3}^{(1)} =$}};
		\ar@{} (0,0);(75, -6) *\cir<4pt>{} = "C"
		\ar@{-} "C";(65, 0)  *\cir<4pt>{} = "E";
		\ar@{-} "C";(85, 0)  *\cir<4pt>{} = "F";
		\ar@{-} "C";(75, -24) *\cir<4pt>{} = "D";
		\ar@{-} "C";(65, -15) *\cir<4pt>{} = "G";
		\ar@{-} "C";(85, -15) *\cir<4pt>{} = "H";
		\ar@{-} "G";"H";
		\ar@{-} "D";"H";
		\ar@{-} "G";"D";
	\end{xy}
	
	$\ $ \\
	\item Let $a = 1, b = 2, c = 3$ and $d = 2$. Then 
	$\ $ \\
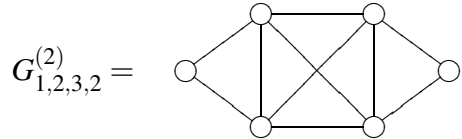
	\begin{xy}
		\ar@{} (0,0);(45,-12.5) *{\text{$G_{1, 2, 3, 2}^{(2)} =$}};
		\ar@{} (0,0);(70, -5) *\cir<4pt>{} = "A"
		\ar@{-} "A";(70, -20) *\cir<4pt>{} = "E";
		\ar@{-} "A";(60, -12.5) *\cir<4pt>{} = "F";
		\ar@{-} "A";(85, -5) *\cir<4pt>{} = "B";
		\ar@{-} "A";(85, -20) *\cir<4pt>{} = "D";
		\ar@{-} "B";(95, -12.5) *\cir<4pt>{} = "C";
		\ar@{-} "B";"D";
		\ar@{-} "B";"E";
		\ar@{-} "C";"D";
		\ar@{-} "D";"E";
		\ar@{-} "E";"F";
	\end{xy}
	
	$\ $ \\
	\item Let $a = 1, b = 3, c = 4$ and $d = 5$. Then \\
	$\ $ \\
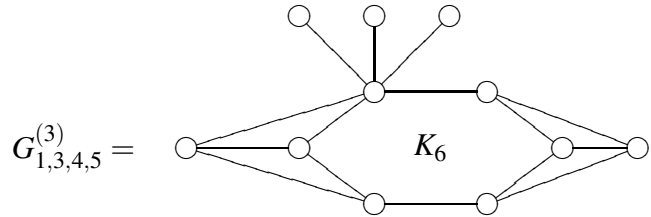
	\begin{xy}
		\ar@{} (0,0);(25,-12.5) *{\text{$G_{1, 3, 4, 5}^{(3)} =$}};
		\ar@{} (0,0);(65, -5) *\cir<4pt>{} = "A"
		\ar@{} "A";(65, -20) *\cir<4pt>{} = "E";
		\ar@{-} "A";(55, -12.5) *\cir<4pt>{} = "F";
		\ar@{-} "A";(80, -5) *\cir<4pt>{} = "B";
		\ar@{} "A";(80, -20) *\cir<4pt>{} = "D";
		\ar@{-} "B";(90, -12.5) *\cir<4pt>{} = "C";
		\ar@{-} "A";(40, -12.5) *\cir<4pt>{} = "G";
		\ar@{-} "C";(100, -12.5) *\cir<4pt>{} = "H";
		\ar@{-} "H";"B";
		\ar@{-} "H";"D";
		\ar@{-} "G";"E";
		\ar@{-} "G";"F";
		\ar@{-} "C";"D";
		\ar@{-} "D";"E";
		\ar@{-} "E";"F";
		\ar@{} (0,0);(72.5,-12.5) *{\text{$K_{6}$}};
		\ar@{-} "A";(75, 5) *\cir<4pt>{};
		\ar@{-} "A";(55, 5) *\cir<4pt>{};
		\ar@{-} "A";(65, 5) *\cir<4pt>{};
	\end{xy}
	
	$\ $ \\
	\item Let $a = 2, b = c = 3$ and $d = 2$. Then \\
	$\ $ \\
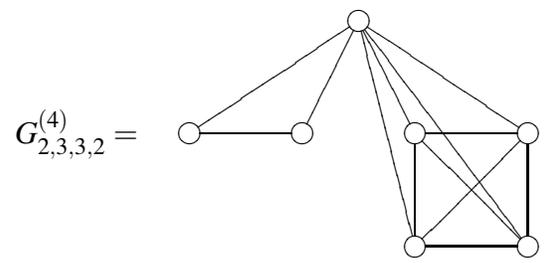
	\begin{xy}
		\ar@{} (0,0);(25,-5) *{\text{$G_{2, 3, 3, 2}^{(4)} =$}};
		\ar@{} (0,0);(70, -5) *\cir<4pt>{} = "A"
		\ar@{-} "A";(70, -20) *\cir<4pt>{} = "E";
		\ar@{-} "A";(85, -5) *\cir<4pt>{} = "B";
		\ar@{-} "A";(85, -20) *\cir<4pt>{} = "D";
		\ar@{-} "B";"D";
		\ar@{-} "B";"E";
		\ar@{-} "D";"E";
		\ar@{} (0,0);(40, -5) *\cir<4pt>{} = "C";
		\ar@{-} "C";(55, -5) *\cir<4pt>{} = "F";
		\ar@{} (0,0);(62.5, 10) *\cir<4pt>{} = "G";
		\ar@{-} "A";"G";
		\ar@{-} "B";"G";
		\ar@{-} "C";"G";
		\ar@{-} "D";"G";
		\ar@{-} "E";"G";
		\ar@{-} "F";"G";
	\end{xy}
	
	$\ $ \\
	\item Let $a = 2, b = c = 3$ and $d = 4$. Then \\
	$\ $ \\
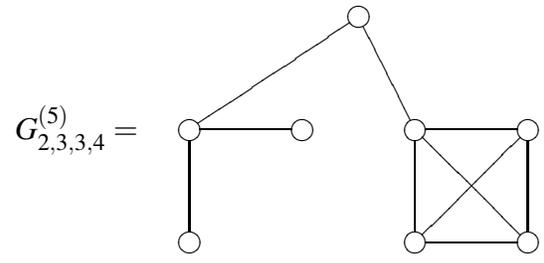
	\begin{xy}
		\ar@{} (0,0);(25,-5) *{\text{$G_{2, 3, 3, 4}^{(5)} =$}};
		\ar@{} (0,0);(70, -5) *\cir<4pt>{} = "A"
		\ar@{-} "A";(70, -20) *\cir<4pt>{} = "E";
		\ar@{-} "A";(85, -5) *\cir<4pt>{} = "B";
		\ar@{-} "A";(85, -20) *\cir<4pt>{} = "D";
		\ar@{-} "B";"D";
		\ar@{-} "B";"E";
		\ar@{-} "D";"E";
		\ar@{} (0,0);(40, -5) *\cir<4pt>{} = "C";
		\ar@{-} "C";(55, -5) *\cir<4pt>{} = "F";
		\ar@{} (0,0);(62.5, 10) *\cir<4pt>{} = "G";
		\ar@{-} "A";"G";
		\ar@{} "B";"G";
		\ar@{-} "C";"G";
		\ar@{} "D";"G";
		\ar@{} "E";"G";
		\ar@{} "F";"G";
		\ar@{-} "C";(40, -20) *\cir<4pt>{};
	\end{xy}
	
	$\ $ \\
	\item Let $a = 3, b = 4, c = 6$ and $d = 5$. Then \\
	$\ $ \\
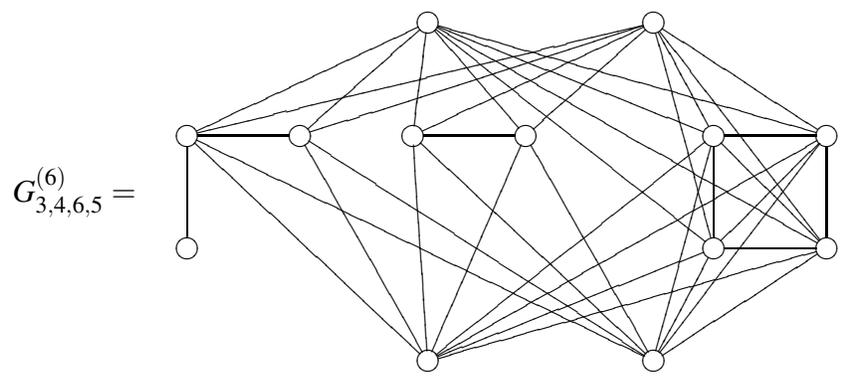
	\begin{xy}
		\ar@{} (0,0);(15,-12.5) *{\text{$G_{3, 4, 6, 5}^{(6)} =$}};
		\ar@{} (0,0);(100, -5) *\cir<4pt>{} = "A"
		\ar@{-} "A";(100, -20) *\cir<4pt>{} = "E";
		\ar@{-} "A";(115, -5) *\cir<4pt>{} = "B";
		\ar@{-} "A";(115, -20) *\cir<4pt>{} = "D";
		\ar@{-} "B";"D";
		\ar@{-} "B";"E";
		\ar@{-} "D";"E";
		\ar@{} (0,0);(60, -5) *\cir<4pt>{} = "C";
		\ar@{-} "C";(75, -5) *\cir<4pt>{} = "F";
		\ar@{} (0,0);(62, 10) *\cir<4pt>{} = "G";
		\ar@{-} "A";"G";
		\ar@{-} "B";"G";
		\ar@{-} "C";"G";
		\ar@{-} "D";"G";
		\ar@{-} "E";"G";
		\ar@{-} "F";"G";
		\ar@{-} "G";(30, -5) *\cir<4pt>{} = "I";
		\ar@{-} "I";(45, -5) *\cir<4pt>{} = "J";
		\ar@{-} "J";"G";
		\ar@{} (0,0);(92, 10) *\cir<4pt>{} = "H";
		\ar@{-} "A";"H";
		\ar@{-} "B";"H";
		\ar@{-} "C";"H";
		\ar@{-} "D";"H";
		\ar@{-} "E";"H";
		\ar@{-} "F";"H";
		\ar@{-} "I";"H";
		\ar@{-} "J";"H";
		\ar@{} (0,0);(62, -35) *\cir<4pt>{} = "K";
		\ar@{-} "A";"K";
		\ar@{-} "B";"K";
		\ar@{-} "C";"K";
		\ar@{-} "D";"K";
		\ar@{-} "E";"K";
		\ar@{-} "F";"K";
		\ar@{-} "I";"K";
		\ar@{-} "J";"K";
		\ar@{} (0,0);(92, -35) *\cir<4pt>{} = "L";
		\ar@{-} "A";"L";
		\ar@{-} "B";"L";
		\ar@{-} "C";"L";
		\ar@{-} "D";"L";
		\ar@{-} "E";"L";
		\ar@{-} "F";"L";
		\ar@{-} "I";"L";
		\ar@{-} "J";"L";
		\ar@{-} "I";(30, -20) *\cir<4pt>{};
	\end{xy}

	$\ $ \\
	\item Let $a = 3, b = 4, c = 5$ and $d = 4$. Then \\
	$\ $ \\
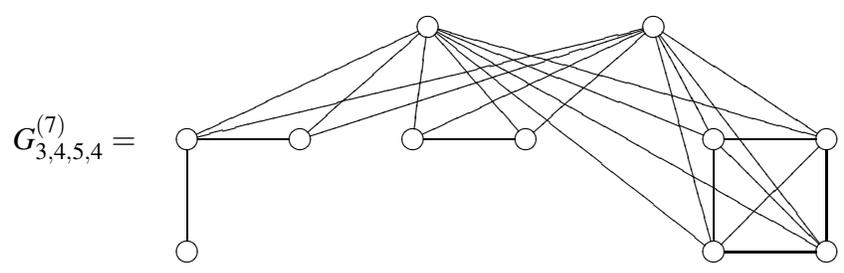
	\begin{xy}
		\ar@{} (0,0);(15,-5) *{\text{$G_{3, 4, 5, 4}^{(7)} =$}};
		\ar@{} (0,0);(100, -5) *\cir<4pt>{} = "A"
		\ar@{-} "A";(100, -20) *\cir<4pt>{} = "E";
		\ar@{-} "A";(115, -5) *\cir<4pt>{} = "B";
		\ar@{-} "A";(115, -20) *\cir<4pt>{} = "D";
		\ar@{-} "B";"D";
		\ar@{-} "B";"E";
		\ar@{-} "D";"E";
		\ar@{} (0,0);(60, -5) *\cir<4pt>{} = "C";
		\ar@{-} "C";(75, -5) *\cir<4pt>{} = "F";
		\ar@{} (0,0);(62, 10) *\cir<4pt>{} = "G";
		\ar@{-} "A";"G";
		\ar@{-} "B";"G";
		\ar@{-} "C";"G";
		\ar@{-} "D";"G";
		\ar@{-} "E";"G";
		\ar@{-} "F";"G";
		\ar@{-} "G";(30, -5) *\cir<4pt>{} = "I";
		\ar@{-} "I";(45, -5) *\cir<4pt>{} = "J";
		\ar@{-} "J";"G";
		\ar@{} (0,0);(92, 10) *\cir<4pt>{} = "H";
		\ar@{-} "A";"H";
		\ar@{-} "B";"H";
		\ar@{-} "C";"H";
		\ar@{-} "D";"H";
		\ar@{-} "E";"H";
		\ar@{-} "F";"H";
		\ar@{-} "I";"H";
		\ar@{-} "J";"H";
		\ar@{-} "I";(30, -20) *\cir<4pt>{};
	\end{xy}

\end{enumerate}
\end{Example}

\bigskip

\noindent
{\bf Acknowledgment.}
The second author was partially supported by JSPS KAKENHI 17K14165.

\end{document}